\documentclass[10pt]{amsart}
\usepackage{latexsym, amsmath, amsfonts, amssymb, mathrsfs, fancyhdr, tikz, tikz-cd}
\usepackage{verbatim}
\usetikzlibrary{matrix,arrows,decorations.pathmorphing,calc}

\def\ds{\displaystyle}
\def\nin{\not \in}
\def\n{\noindent}
\def\R{\mathbb{R}}
\def\Rpq{\mathbb{R}^{p,q}}
\def\eij{e_{ij}}

\def\e{\varepsilon}
\def\s{\sigma}

\def\d{\delta}
\def\Tone{\mathcal{T}^{(1)}}
\def\T{\mathcal{T}}
\def\l{\lambda}
\def\En{\mathbb{E}^n}

\def\X{\mathcal{X}}
\def\V{\mathcal{V}}

\def\E{\mathbb{E}}
\def\a{\alpha}

\def\k{\kappa}
\def\t{\tau}

\def\N{\mathbb{N}}
\def\la{\langle}
\def\ra{\rangle}

\newtheorem{theorem}{Theorem}

\newtheorem{cor}[theorem]{Corollary}

\theoremstyle{definition}

\newtheorem{remark}{Remark}

\theoremstyle{remark}

\numberwithin{equation}{section}

\begin{document}

\title{Approximating continuous maps by isometries}

\author{Barry Minemyer}
\address{Department of Mathematics, The Ohio State University, Columbus, Ohio 43210}
\email{minemyer.1@osu.edu}



\date{\today.}



\begin{abstract}
The Nash-Kuiper Theorem states that the collection of $C^1$-isometric embeddings from a Riemannian manifold $M^n$ into $\E^N$ is $C^0$-dense within the collection 
of all smooth 1-Lipschitz embeddings provided that $n < N$.  
This result is now known to be a consequence of Gromov's more general $h$-principle.  
There have been some recent extensions of the Nash-Kuiper Theorem to Euclidean polyhedra, 
which in some sense provide a very specialized discretization of the $h$-principle.
In this paper we will discuss these recent results 
and provide generalizations 
to the setting of isometric embeddings of spaces endowed with indefinite metrics into Minkowski space.
The new observation is that, when dealing with Minkowski space, the assumption ``1-Lipschitz" can be removed.  
Thus, we obtain results about isometric embeddings that are $C^0$-dense within the collection of {\it all} continuous maps.
\end{abstract}

\maketitle



\section{Introduction}\label{section 1}

Let $(M^m,g)$ denote an $m$-dimensional Riemannian manifold.  
The famous Nash-Kuiper Theorem (\cite{Nash1}, \cite{Kuiper}) states that any smooth 1-Lipschitz embedding $f: (M^m, g) \rightarrow \En$ 
is $\e$-close to a $C^1$-isometric embedding for any $\e > 0$ provided $n > m$.  
Here, two maps $f, f^\prime : M \rightarrow \En$ are {\it $\e$-close} if $|f(x) - f^\prime (x)| < \e$ for all $x \in M$, which is sometimes also stated as {\it $C^0$-close}.
In other words, the Nash-Kuiper Theorem states that the collection of $C^1$-isometric embeddings is {\it $C^0$-dense} in the collection of all smooth 1-Lipschitz embeddings of $M$ into $\En$, 
provided that you have at least one degree of codimension.

When this result was first published by Nash in 1954 (in the case $m \leq n-2$) it was stunning to many mathematicians.  
This was due to the general ``flexibility" of $C^1$-isometric embeddings when compared to the known rigidity of $C^k$, $k \geq 2$, isometric embeddings. 
This is now known to be a specific consequence of Gromov's much more general $h$-principle, 
popularized by Gromov in \cite{Gromov PDR} and eloquently explained by Eliashberg and Mishachev in \cite{EM}.
In \cite{Gromov PDR} and \cite{Gromov green} Gromov used the $h$-principle to prove that any strictly short map between $n$-manifolds is $C^0$-close to a $C^0$-path isometry 
(i.e., a continuous map that preserves the length of paths).
So one sees that the necessity of having {\it any} codimension can be removed if we sacrifice the property of being an embedding (and one degree of differentiability).

A {\it Euclidean polyhedron} (or {\it polyhedral space}) is a metric space $\X$ equipped with a locally finite simplicial triangulation $\T$ such that 
every $k$-dimensional simplex of $\T$ is affinely isometric to a simplex in Euclidean space $\E^k$ (for all $k$).  
Note that, due to the triangulation being locally finite, all Euclidean polyhedra are proper (meaning that closed bounded sets are compact) and thus are geodesic metric spaces.
Such spaces clearly are not necessarily topological manifolds, so in some sense they are generalizations of manifolds.  
But they have the added bonus of the metric being flat when restricted to any simplex, so in that sense they are nicer than Riemannian manifolds.
In any case, any Riemannian manifold can be obtained as a ``nice" inverse limit of Euclidean polyhedra (see any of \cite{BBI}, \cite{Petrunin}, \cite{Minemyer3}).

In the same text where Gromov develops the $h$-principle \cite{Gromov PDR} he asks whether or not Euclidean polyhedra admit piecewise-linear isometries
into the same dimensional Euclidean space.  
Such a result would lead to a pl-analogue to Gromov's result above concerning the approximation of 1-Lipschitz maps between manifolds by isometries.
This question was answered in the affirmative by Zalgaller \cite{Zalgaller} and Krat \cite{Krat}, the former of which was the original motivation for Gromov's question.
In the spirit of the $h$-principle though, Krat asked if such pl isometries are $C^0$-dense within the collection of all 1-Lipschitz maps.  
She proved this result in \cite{Krat} for the case when $n=2$, and the result was generalized to all dimensions by Akopyan in \cite{Akopyan}.  
The case of pl isometric embeddings was originally considered in the case when $n=2$ by Burago and Zalgaller in \cite{BZ}, and recently considered by the author for all dimensions in \cite{Minemyer1}.  

The necessity of the assumption that all of the maps be ``1-Lipschitz" in the preceding results is clear.  
In Euclidean space there is no way to approximate a long path by a short path. 
But the reverse statement is clearly possible by approximating a short path by a much longer ``polygonal" path (see Figure \ref{seventhfig} below).  
If the target Euclidean space is replaced by Minkowski space $\Rpq$ though, then there is hope of removing this assumption.
In particular, the collection of pl path isometries (respectively isometric embeddings) may be $C^0$-dense within the collection of {\it all} continuous maps.

An {\it indefinite metric polyhedron} is a triple $(\X, \T, g)$ where $\X$ is a topological space, $\T$ is a simplicial triangulation of $\X$, and $g$ is a function that assigns a real number to every edge of $\T$.  
This edge function $g$ naturally associates to each $k$-dimensional simplex in $\T$ a unique quadratic form on $\R^k$, and in turn this assigns a unique indefinite metric structure to all of $\X$.
Note that these quadratic forms need not be positive definite nor even non-degenerate, but if all of these associated quadratic forms are positive definite then this just leads to a Euclidean polyhedron.
So in particular the class of indefinite metric polyhedra contains the class of Euclidean polyhedra as the special case when the quadratic form defined on every simplex is positive-definite.

Let $(\X, \T, g)$ be an indefinite metric polyhedron, and let $G$ denote the quadratic form determined by $g$.  
Let $f: \X \rightarrow \R^{p,q}$ be any continuous function.  
The map $f$ determines a unique indefinite metric $g_f$ on $(\X, \T)$ 
and this indefinite metric induces a quadratic form $G_f$ on each simplex of $\T$ as discussed above (please see Section \ref{section: preliminaries} for more details).  
We call $G_f$ the {\it induced quadratic form} of $f$.  
We say that $f$ is a {\it piecewise linear isometry} (or {\it pl isometry}) of $\X$ into $\Rpq$ if $f$ is piecewise linear (meaning that it is simplicial on some subdivision of $\T$) and if $G = G_f$ on all simplices in a subdivision of $\T$ on which $f$ is simplicial.  
The map $f$ is a {\it pl isometric embedding} if in addition to being a pl isometry it is also an embedding.

There have been some very recent results concerning {\it simplicial} isometric embeddings of indefinite metric polyhedra into Minkowski space $\R^{p,q}$ (see \cite{Minemyer2} and \cite{GZ}).  
These simplicial isometric embeddings require a high degree of codimension, and in that sense resemble the rigidity of $C^k$ isometric embeddings ($k>1$) of Riemannian manifolds into Euclidean space.
But what if we allow for piecewise-linear maps instead of simplicial?  
In this setting we can combine a Theorem due to Krat/Akopyan (Theorem \ref{Akopyan} in Section \ref{section: preliminaries}) with a few geometric tricks to prove the following Theorem.

\vskip 10pt

\begin{theorem}\label{thm:isometric embedding Minkowski space}
Let $(\X, \T, g)$ be an $n$-dimensional indefinite metric polyhedron with vertex set $\V$, 
and let $\{ \e_i \}_{i = 1}^{\infty}$ be a sequence of positive real numbers. 
Let $f: \X \rightarrow \R^{p,q}$ be a continuous function where $p \geq n$, $q \geq n$, and $p + q \geq 3n$, and fix a vertex $v \in \V$.
Then there exists a piecewise linear isometric embedding $h: \X \rightarrow \R^{p,q}$ such that for any $k \in \N$ and for any $x \in Sh^k(v)$, $|f(x) - h(x)| < \e_k$.  
\end{theorem}

\vskip 2pt

In particular, if one lets $\e_k = \e$ for all $k$, then one obtains as a Corollary:

\vskip 10pt

\begin{cor}\label{cor:no shell notation}
Let $(\X, \T, g)$ be an $n$-dimensional indefinite metric polyhedron, let $\e > 0$, and let $f: \X \rightarrow \R^{p,q}$ be a continuous function where $p \geq n$, $q \geq n$, and $p + q \geq 3n$.
Then there exists a piecewise linear isometric embedding $h: \X \rightarrow \R^{p,q}$ such that $|f(x) - h(x)| < \e$.  
\end{cor}

\vskip 2pt

So we see that the collection of pl isometric embeddings is $C^0$-dense within the collection of {\it all} continuous functions (provided that we have the codimension requirements listed in the Theorem).  
The notation ``$Sh^k(v)$" from Theorem \ref{thm:isometric embedding Minkowski space} will be defined in Section \ref{section: preliminaries}, but its purpose is simply to allow the $\e$ from Corollary \ref{cor:no shell notation} to taper to 0 as one moves further away from some fixed point $v$.
Lastly, note that these codimension requirements are likely not optimal, and it may be possible that one could obtain bounds as low as $p+q \geq 2n+1$.

An immediate Corollary of the proof of Theorem \ref{thm:isometric embedding Minkowski space} is the following:

\vskip 10pt

\begin{cor}\label{cor:isometry Minkowski space}
Let $(\X, \T, g)$ be an $n$-dimensional indefinite metric polyhedron with vertex set $\V$, 
and let $\{ \e_i \}_{i = 1}^{\infty}$ be a sequence of positive real numbers. 
Let $f: \X \rightarrow \R^{p,q}$ be a continuous function where both $p,q \geq n$, and fix a vertex $v \in \V$.
Then there exists a piecewise linear isometry $h: \X \rightarrow \R^{p,q}$ such that for any $k \in \N$ and for any $x \in Sh^k(v)$, $|f(x) - h(x)| < \e_k$.
\end{cor}

\vskip 2pt

Isometric embeddings of manifolds into Minkowski space have been studied to some extent by Greene in \cite{Greene} and Gromov-Rokhlin in \cite{GR}.  
But neither of these publications considered the existence of such maps from a ``$C^0$-dense" standpoint.  
Essentially the same proof as that of Theorem \ref{thm:isometric embedding Minkowski space}, 
but by replacing Krat/Akopyan's Theorem \ref{Akopyan} by the Nash-Kuiper Theorem, proves:

\vskip 10pt

\begin{theorem}\label{theorem:Nash-Kuiper into Minkowski space}
Let $M$ be an $n$-dimensional manifold, let $g$ be a smooth metric tensor of any signature on $M$, 
and let $f:M \rightarrow \R^{p,q}$ be any continuous map with both $p,q \geq 2n$.  
Then for any $\e > 0$ there exists a $C^1$-isometric embedding $h:M \rightarrow \R^{p,q}$ such that $|f(x) - h(x)| < \e$ for all $x \in M$.  
That is, $h$ is $C^0$-close to $f$.  
\end{theorem}

\vskip 2pt

Note that in Theorem \ref{theorem:Nash-Kuiper into Minkowski space} there are absolutely no conditions on the signature of the metric $g$.  
In particular, $g$ could be degenerate.

\begin{remark}
The results in this paper were developed during the author's work in \cite{Minemyer4}.  
These results ended up not being used in \cite{Minemyer4}, but the author felt that they were interesting in their own right.  
The proof's are not too difficult though and could even be considered applications of Krat/Akopyan's Theorem \ref{Akopyan} and the Nash-Kuiper $C^1$-isometric embedding Theorem.  
The author's opinion is that the results stated here are more interesting than the techniques used in the proofs.
\end{remark}

\begin{remark}
Even though Theorem \ref{thm:isometric embedding Minkowski space}, Corollary \ref{cor:isometry Minkowski space}, and Theorem \ref{theorem:Nash-Kuiper into Minkowski space}
above deal with maps into Minkowski space $\R^{p,q}$, the metric on the set of functions is always defined using the Euclidean metric on $\R^{p+q}$.  
To avoid confusion, in this paper the use of straight brackets $| \cdot |$ will always denote the Euclidean norm.  
\end{remark}

This paper is ordered as follows.  
In Section \ref{section: preliminaries} we discuss an array of preliminary topics, including Akopyan's Theorem \ref{Akopyan}, Minkowski space, and quadratic forms associated to indefinite metric polyhedra.  
Then  in Section \ref{section: proofs} we prove 
Theorem \ref{thm:isometric embedding Minkowski space}, Corollary \ref{cor:isometry Minkowski space}, and Theorem \ref{theorem:Nash-Kuiper into Minkowski space}.


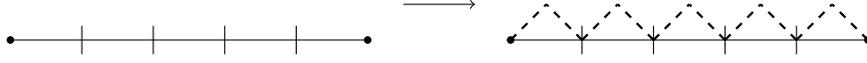
\begin{figure}
\begin{center}
\begin{tikzpicture}[scale=0.95]

\draw (0,0) -- (5,0);
\draw[fill=black!] (0,0) circle (0.3ex);
\draw[fill=black!] (5,0) circle (0.3ex);
\draw (1,-0.2) -- (1,0.2);
\draw (2,-0.2) -- (2,0.2);
\draw (3,-0.2) -- (3,0.2);
\draw (4,-0.2) -- (4,0.2);

\draw[->] (5.5,0.5) -- (6.5,0.5);

\draw (7,0) -- (12,0);
\draw[line width=0.3mm,dashed] (7,0) -- (7.5,0.5);
\draw[line width=0.3mm,dashed] (8,0) -- (7.5,0.5);
\draw[line width=0.3mm,dashed] (8,0) -- (8.5,0.5);
\draw[line width=0.3mm,dashed] (9,0) -- (8.5,0.5);
\draw[line width=0.3mm,dashed] (9,0) -- (9.5,0.5);
\draw[line width=0.3mm,dashed] (10,0) -- (9.5,0.5);
\draw[line width=0.3mm,dashed] (10,0) -- (10.5,0.5);
\draw[line width=0.3mm,dashed] (11,0) -- (10.5,0.5);
\draw[line width=0.3mm,dashed] (11,0) -- (11.5,0.5);
\draw[line width=0.3mm,dashed] (12,0) -- (11.5,0.5);
\draw[fill=black!] (7,0) circle (0.3ex);
\draw[fill=black!] (12,0) circle (0.3ex);
\draw (11,-0.2) -- (11,0.2);
\draw (8,-0.2) -- (8,0.2);
\draw (9,-0.2) -- (9,0.2);
\draw (10,-0.2) -- (10,0.2);


\end{tikzpicture}
\end{center}
\caption{Approximating a short path by a nearby longer path (dashed).}
\label{seventhfig}
\end{figure}

\vskip 20pt

\section{Minkowski space, quadratic forms, and the Krat/Akopyan Theorem}\label{section: preliminaries}

\subsection{Minkowski space $\Rpq$}

\emph{Minkowski space of signature $(p, q)$}, denoted by $\R^{p,q}$, is $\R^{p + q}$ endowed with the symmetric bilinear form of signature $(p,q)$.  More specifically, if $\vec{v}, \vec{w} \in \R^{p,q}$ with $\vec{v} = (v_i)_{i = 1}^{p + q}$ and $\vec{w} = (w_i)_{i = 1}^{p + q}$, then
	\begin{equation*}
	\langle \vec{v} , \vec{w} \rangle_{\R^{p,q}} := \langle \vec{v} , \vec{w} \rangle := \sum_{i = 1}^p v_i w_i - \sum_{j = p + 1}^{p + q} v_j w_j .
	\end{equation*}

The use of $\Rpq$ will specifically mean $\R^{p + q}$ endowed with the symmetric bilinear form of signature $(p,q)$, $\mathbb{E}^N$ will mean $\R^N$ with the symmetric bilinear form of signature $(N,0)$, and $\R^N$ will mean to include the possibility of \emph{any} Minkowski inner product of signature $(p',q')$ such that $p' + q' = N$.

\vskip 20pt

\subsection{Quadratic forms associated to indefinite metric polyhedra}

Let $(\X, \T, g)$ be an {\it indefinite metric polyhedron}.  
This just means that $\X$ is a topological space, $\T$ is a locally finite simplicial triangulation of $\X$, and $g$ is a function which assigns a real number to each edge of $\T$.
This function $g$ defines a unique indefinite metric over each simplex $\s \in \T$, and thus over all of $\X$, as follows.

Let $\s = \langle v_0, v_1, ..., v_k \rangle \in \T$ be a $k$-dimensional simplex.
Embed $\s$ into $\R^k$ by identifying $v_0$ with the origin, and for $1 \leq i \leq k$ identifying $v_i$ with the terminal point of the $i^{th}$ standard basis vector.
Let $\vec{w}_i := v_i - v_0$ denote the $i^{th}$ standard basis vector, and let $e_{ij}$ denote the edge in $\s$ between the vertices $v_i$ and $v_j$.  

The indefinite metric $g$ (and our choice of ordering of the vertices of $\s$) defines a quadratic form $G$ on $\R^k$ as follows.  Define
	\begin{align*}
	G(w_i) &= s(g (e_{0i}))  \\
	G(w_i - w_j) &= s(g(e_{ij}))
	\end{align*}
where
	\begin{equation*}
	\ds{s(x) = \left\{ \begin{array}{rl} x^2 & \quad \text{if } x \geq 0 \\ -x^2 & \quad \text{if } x < 0 \end{array} \right.  }
	\end{equation*} 
is the {\it signed squared} function.
Let $\la , \ra_g$ denote the symmetric bilinear form associated to $G$.  
A simple calculation, worked out in \cite{Minemyer2}, shows that
	\begin{equation}\label{definition of quadratic form}
	\langle \vec{w}_i , \vec{w}_j \rangle_g \, = \frac{1}{2} \left( G(\vec{w}_i) + G(\vec{w}_j) - G(\vec{w}_i - \vec{w}_j) \right).   
	\end{equation}
So $G$ is completely determined by the above definition, which is sometimes called the {\it polarization identity} of $G$.  
We will abuse notation and refer to $G$ as a quadratic form on $\s$, when rigorously $G$ is really a quadratic form on $\R^k$.  

Given a quadratic form $G$ on $\s$ as above, define the {\it energy} of an edge $e$ to simply be $G(e)$.  
Equation \eqref{definition of quadratic form} shows that a quadratic form is uniquely determined by the energy that it assigns to each edge.  
Thus, the set of quadratic forms on a $k$-dimensional simplex $\s$ can naturally be identified with $\R^n$ where $n = $  $k+1 \choose 2$.  
Each coordinate in $\R^n$ is naturally parameterized by the energy of the corresponding edge of $\s$.  

Now let $f: \X \rightarrow \R^{p,q}$ be any continuous function, where $\R^{p,q}$ denotes Minkowski space of signature $(p,q)$.  
Let $\s$ be as above.  
The map $f$ determines a unique indefinite metric $g_f$ on $(\X, \T)$ by defining
	\begin{equation}\label{definition of induced metric}
	g_f (e_{ij}) := \la f(v_i) - f(v_j) , f(v_i) - f(v_j) \ra
	\end{equation}
where $v_i$ and $v_j$ are the vertices incident with $e_{ij}$, and where $\la , \ra$ is the Minkowski bilinear form on $\Rpq$.  
The indefinite metric $g_f$ induces a quadratic form $G_f$ on $\R^k$ just as above, called the {\it induced quadratic form} of $f$.  
The map $f$ is a {\it simplicial isometry} if it is simplicial over $\T$ (meaning that it is linear on each simplex of $\T$) and if it satisfies that $G_f(\s) = G(\s)$ for all $\s \in \T$.  
We say that $f$ is a {\it piecewise linear isometry} (or {\it pl isometry}) of $\X$ into $\Rpq$ if $f$ is piecewise linear (meaning that it is simplicial on some subdivision $\T'$ of $\T$) and if is a simplicial isometry with respect to $\T'$.  
The map $f$ is a {\it pl isometric embedding} (respectively a {\it simplicial isometric embedding}) if in addition to being a pl isometry (respectively a simplicial isometry) it is also an embedding.

We say that an indefinite metric polyhedron $(\X, \T, g)$ is {\it Euclidean} if the quadratic form $G(\s)$ induced by $g$ on $\s \in \T$ is positive definite for all $\s \in \T$.  
So in some sense, Euclidean polyhedra are combinatorial analogues to Riemannian manifolds.
It is well known that the collection of positive definite quadratic forms is closed under addition and positive scalar multiplication.  
Thus, they form an open cone within the collection of all indefinite metric polyhedra, an observation which was also pointed out by Rivin in \cite{Rivin}.

\vskip 20pt

\subsection{Splitting of $G_f$}
Let $f: \X \rightarrow \Rpq$ be a simplicial map.  
Write $f = f_1 \oplus f_2$ where the ``$\oplus$" denotes the {\it concatenation} of $f_1$ and $f_2$.
So $f_1 : \X \rightarrow \R^a$ and $f_2 : \X \rightarrow \R^b$ for some integers $a$ and $b$ where $a + b = p + q$.  
Let $e_{ij}$ denote the edge between vertices $v_i$ and $v_j$.  
Then, using superscripts to denote the component functions of $f$, $f_1$, and $f_2$:
	\begin{align*}
	g_f^2 (e_{ij}) &= \la f(v_i) - f(v_j), f(v_i) - f(v_j) \ra  \\
	&= \sum_{k = 1}^{p+q} \eta(k) (f^k(v_i) - f^k(v_j))^2  \\
	&= \sum_{k=1}^a \eta(k) (f^k_1 (v_i) - f^k_1(v_j))^2 + \sum_{k=a+1}^{a+b} \eta(k) (f^k_2 (v_i) - f^k_2(v_j))^2  \\
	&= g_1^2 (e_{ij}) + g_2^2 (e_{ij})
	\end{align*}
where $\eta(k) = \pm 1$ depending on the respective coordinate, and where $g_1$ and $g_2$ denote the indefinite metrics induced by $f_1$ and $f_2$, respectively.

Combining the above with equations \eqref{definition of quadratic form} and \eqref{definition of induced metric} yields
	\begin{equation}\label{splitting of quadratic form}
	G_f = G_f^1 + G_f^2
	\end{equation}
where $G_f^1$ and $G_f^2$ are the quadratic forms induced by $f_1$ and $f_2$, respectively.

\vskip 20pt

\subsection{Akopyan's Theorem}  In this Subsection we provide some necessary terminology and then formally state Akopyan's result, which is the key ingredient in proving Theorem \ref{thm:isometric embedding Minkowski space} and Corollary \ref{cor:isometry Minkowski space}.  
The statement provided here is slightly more general than what is in \cite{Akopyan}, but only applies to {\it Euclidean} polyhedra.  
The proof goes through nearly unchanged, and can be found in \cite{Akopyan} (in Russian).
An English proof can be found in \cite{Minemyer thesis}, and the case when $n=2$ can be found in \cite{PY}.

Let $(\X, \T)$ be a polyhedron (that is, a topological space $\X$ with a locally finite triangulation $\T$) and let $x \in \X$.  
For a vertex $v$, the closed star of $v$ will be denoted by $St(v)$.  
We define $St^2(v) := \bigcup_{u \in St(v)} St(u)$ and for any $k \in \mathbb{N}$ we recursively define $St^{k + 1}(v) := \bigcup_{u \in St^k(v)} St(u)$.  
Then define the \emph{$k^{th}$ shell about $x$}, denotes by $Sh^k(x)$, as:
	\begin{enumerate}
	\item $Sh^1(x) = St(x)$
	\item $Sh^k(x) = St^k(x) \setminus St^{k - 1}(x)$ for $k \geq 2$
	\end{enumerate}

Notice that $Sh^k(x) \cap Sh^l(x) = \emptyset$ for $k \neq l$, and that $\bigcup_{k = 1}^{\infty} Sh^k(x) = \X$.  
So the collection of shells partitions $\X$.  
Note that it is certainly possible for $Sh^k(x) = \emptyset$ in the presence of nontrivial homology, in which case $Sh^l(x) = \emptyset$ for all $l \geq k$.  
Also notice that $St^k(x)$ and $Sh^k(x)$ both depend on the triangulation that is being considered.  If the triangulation is to be emphasized, then it will be put as a subscript.  So $St^k_{\T}(x)$ and $Sh^k_{\T}(x)$ denote the $k^{th}$ closed star and the $k^{th}$ shell of $x$ with respect to $\T$, respectively.  

The following Theorem was proved by Krat in \cite{Krat} for the case when $n=2$, and then for general dimensions by Akopyan in \cite{Akopyan}.

\vskip 10pt
	
\begin{theorem}[Krat \cite{Krat}, Akopyan \cite{Akopyan}]\label{Akopyan}
Let $(\X, \T, g)$ be an $n$-dimensional Euclidean polyhedron with vertex set $\V$ and let $\{ \e_i \}_{i = 1}^{\infty}$ be a sequence of positive real numbers converging monotonically to $0$.  Let $f:\X \rightarrow \mathbb{E}^N$ be a short map with $N \geq n$ and fix a vertex $v \in \V$.  Then there exists a pl isometry $h:\X \rightarrow \mathbb{E}^N$ such that for any $k \in \mathbb{N}$ and for any $x \in Sh^k(v)$, $|f(x) - h(x)| < \e_k$.
\end{theorem}

\vskip 8pt

The slight difference between Theorem \ref{Akopyan} and what is contained in \cite{Akopyan} is that Theorem \ref{Akopyan} allows the $\e$-approximation to decrease to zero as you 
move farther and farther away from some fixed point.  
This allows us the cut one dimension off of the codimension requirements in Theorem \ref{thm:isometric embedding Minkowski space}.  
But if in Theorem \ref{thm:isometric embedding Minkowski space} one only requires that $p + q \geq 3n+1$ then Akopyan's original result from \cite{Akopyan} is sufficient.  

\vskip 10pt
	  

\vskip 20pt

\subsection{Akopyan's Theorem in terms of quadratic forms}
Let $P$ and $Q$ denote two quadratic form on $\R^k$.  
Recall that the notation $P < Q$ means that $P(v) < Q(v)$ for all $v \in \R^k$, and similarly for $\leq$.  
Given an indefinite metric polyhedron $(\X, \T, g)$ and a simplicial map $f: \X \rightarrow \Rpq$, 
we say that $f$ is {\it short}, or {\it 1-Lipschitz}, if $G_f \leq G$ on every simplex of $\T$, and $f$ is {\it strictly short} if $G_f < G$ for all simplices in $\T$.  
Note that, if $\X$ is a Euclidean polyhedron, then this definition of 1-Lipschitz is equivalent to the usual definition for a metric space.
This definition is also equivalent to how we used the term ``short" in the Introduction and in Krat/Akopyan's Theorem \ref{Akopyan}, but is now slightly generalized to include indefinite metrics.

When proving Theorem \ref{thm:isometric embedding Minkowski space} it will be useful to have a version of Krat/Akopyan's Theorem \ref{Akopyan} for negative-definite metrics.
The next statment is just a reworded version of Theorem \ref{Akopyan} for the negative-definite setting.

\vskip 10pt
	
\begin{theorem}[Krat/Akopyan's Theorem for negative-definite polyhedra]\label{negative-definite Akopyan}
Let $(\X, \T, g)$ be an $n$-dimensional indefinite metric polyhedron with vertex set $\V$ and associated quadratic form $G$.  
Let $f:\X \rightarrow \mathbb{R}^{0,N}$ be a continuous map with associated quadratic form $G_f$.
Assume that $G_f \geq G$ (which necessarily implies that $G \leq 0$, i.e. that $G$ is negative-definite).
Let $\{ \e_i \}_{i = 1}^{\infty}$ be a sequence of positive real numbers, assume $N \geq n$, and fix a vertex $v \in \V$. 
Then there exists a pl isometry $h:\X \rightarrow \mathbb{R}^{0,N}$ such that for any $k \in \mathbb{N}$ and for any $x \in Sh^k(v)$, $|f(x) - h(x)| < \e_k$.
\end{theorem}

\vskip 20pt

\section{Proofs of Theorem \ref{thm:isometric embedding Minkowski space}, Corollary \ref{cor:isometry Minkowski space}, and Theorem \ref{theorem:Nash-Kuiper into Minkowski space}.}\label{section: proofs}

\begin{proof}[Proof of Theorem \ref{thm:isometric embedding Minkowski space}]

Let $(\X, \T, g)$ be an $n$-dimensional indefinite metric polyhedron, and let $N := p + q$.
Since $f$ can be approximated arbitrarily closely by a pl map, by passing to a subdivision of $\T$ (which may be finer and finer as we move away from $v$) we may assume that $f$ is simplicial with respect to $\T$.  

Let $G$ and $G_f$ denote the symmetric bilinear forms determined by the metric $g$ and the function $f$, respectively.  
Write 
	\begin{equation}\label{splitting f}
	f = f^+ \oplus f^* \oplus f^-
	\end{equation}
where
	\begin{align*}
	&f^+ : \X \rightarrow \R^{n,0} \; \text{ with associated quadratic form } G_f^+  \\
	&f^* : \X \rightarrow \R^{p-n,q-n}  \; \text{ with associated quadratic form } G_f^* \\
	&f^- : \X \rightarrow \R^{0,n} \; \text{ with associated quadratic form } G_f^-.
	\end{align*}
By equation \eqref{splitting of quadratic form} we know that $G_f = G_f^+ + G_f^* + G_f^-$.  

Since $p + q \geq 3n$, we have that $(p-n) + (q-n) \geq n$.  
So the target spaces of each of the three maps on the right hand side of equation \eqref{splitting f} contain at least $n$ dimensions.  
By perturbing the vertices of $f(\X)$ into general position one coordinate at a time, we may assume both that $f$ is an embedding 
and that $f^+ \oplus f^*$ is an embedding when restricted to the closed star of any vertex (furthermore called a {\it local embedding}).
For the full details of this argument, please see the proof of Theorem 1.2 (1) from \cite{Minemyer1}.

We now want to construct a quadratic form $H$ on $\T$ that satisfies the following two properties
	\begin{equation}\label{H < G}
	H < G
	\end{equation}
and
	\begin{equation}\label{H < G_f}
	H < G_f .
	\end{equation}
If $\X$ is compact then we simply scale the identity metric on $\X$ (the metric which gives every edge a length of 1) by a large (in absolute value) negative number to obtain $H$.  
If $\X$ is not compact then we fix $v$ in the vertex set of $\T$ and scale the edges in $Sh^k(v)$ sequentially by (possibly) larger and larger negative numbers.
It is possible that, when going from $Sh^k(v)$ to $Sh^{k+1}(v)$, the increase in size of the scaling factor will be too large so that one (or both) of $G-H$ or $G_f - H$ is not positive definite.  
To remedy this, we take a very find subdivision of $Sh^{k+1}(v) \setminus Sh^k(v)$ and gradually increase the scale of the edges as we move away from $Sh^k(v)$.  

Equation \eqref{H < G_f} gives
	\begin{equation*}
	G_f^+ + G_f^* + G_f^- = G_f > H \qquad \Longrightarrow \qquad G_f^- > H - G_f^+ - G_f^*.
	\end{equation*}
So we may apply the negative-definite version of Akopyan's Theorem (Theorem \ref{negative-definite Akopyan}) 
to obtain a pl map $h^- : \X \rightarrow \R^{0,n}$ with associated 
quadratic form $G_h^-$ that satisfies
	\begin{equation}\label{H equality}
	G_h^- = H - G_f^+ - G_f^*
	\end{equation}
over all simplices of some subdivision $\T^\prime$ of $\T$, and is as precise of an approximation to $f^-$ as we require within $Sh^k(v)$.  

To see how precise we need to approximate $f^-$, consider the collection 
	\begin{equation*}
	\{ st(p) | p \in Sh^k_{\T}(v) \}
	\end{equation*} 
where $st(p)$ denotes the open star of $p$ with respect to $\T$.  
Since $\T$ is locally finite, there exists a finite subset of this collection that covers $Sh^k_{\T}(v)$.  
This finite collection has a Lebesgue number which we will denote $\d_k > 0$.  
Let 
	\begin{equation*}
	\Delta_k = \{ (x, x) | x \in Cl(Sh^k_{\T}(v)) \}
	\end{equation*}
denote the diagonal of $Cl(Sh^k_{\T}(v)) \times Cl(Sh^k_{\T}(v))$ (where $Cl$ denotes the \emph{closure}), 
and let $b(\Delta_k, \d_k)$ denote the open neighborhood of radius $\d_k$ of $\Delta_k$.  
Then $b(\Delta_k, \d_k)^C$ is a closed subset of $Cl(Sh^k_{\T}(v)) \times Cl(Sh^k_{\T}(v))$ and is therefore compact.
Consider the function $\psi_k: b(\Delta_k, \d_k)^C \rightarrow \mathbb{R}$ defined by $\psi_k(x,y) := |f(x) - f(y)|_{\mathbb{E}^N}$.
The map $\psi_k$ is positive over all of $b(\Delta_k, \d_k)^C$ since $f$ is an embedding.
Then since $b(\Delta_k, \d_k)^C$ is compact, there exists $\mu_k > 0$ such that $\psi_k(x, y) > \mu_k$ for all $(x, y) \in b(\Delta_k, \d_k)^C$.

We obtain $h^-$ by applying Theorem \ref{negative-definite Akopyan} to $f^-$ with $\e_k := \frac{\mu_k}{3} $ accuracy within $Sh^k_{\T}(v)$.
Let $f^\prime := f^+ \oplus f^* \oplus h^-$.  
By the choice of $\e_k$, $f^\prime(x) \neq f^\prime(y)$ for any $(x, y) \in b(\Delta_k, \d_k)^C$.  
Also, $f^\prime(x) \neq f^\prime(y)$ for any $(x, y) \in b(\Delta_k, \d_k)$ since $f^+ \oplus f^*$ is injective on the $\d_k$ neighborhood of every point.
Thus, this new map $f^\prime$ is still injective.

Now, by equation \eqref{H < G} we have that
	\begin{equation}\label{short inequality}
	G > H = G_f^+ + G_f^* + G_h^-  \qquad \Longrightarrow \qquad  G - G_f^* - G_h^- > G_f^+.
	\end{equation}
In the exact same way as above, we may perturb the vertices of $f^*$ and $h^-$ so that $f^* \oplus h^-$ is a local embedding 
while maintaining both the inequality on the right hand side of equation \eqref{short inequality} and the fact that $f^\prime$ is a global embedding.  

We now apply Theorem \ref{Akopyan} to obtain a map $h^+ : \X \rightarrow \R^{n,0}$ 
with associated quadratic form $G_h^+$ that satisfies
	\begin{equation}\label{isometry}
	G - G_f^* - G_h^- = G_h^+  \qquad \Longrightarrow \qquad G = G_h^+ + G_f^* + G_h^-
	\end{equation}
over all simplices of some subdivision $\T^{\prime \prime}$ of $\T^\prime$.  
Using the exact same argument as above, we can choose $h^+$ to be a close enough approximation to $f^+$ so that the map $h := h^+ \oplus f^* \oplus h^-$ is still an embedding.  
Then by the right hand side of equation \eqref{isometry}, we see that $h$ is the desired isometric embedding which is a suitable approximation of $f$.  
	
\end{proof}

\vskip 20pt

\begin{proof}[Proof of Corollary \ref{cor:isometry Minkowski space}] 

In the proof of Theorem \ref{thm:isometric embedding Minkowski space}, we first apply Theorem \ref{negative-definite Akopyan} to the map $f^-$ 
and then apply Theorem \ref{Akopyan} to the map $f^+$.  
The purpose of $f^*$ is to ensure that we have enough coordinates so that the maps $f^+ \oplus f^*$ and $f^* \oplus h^-$ can be perturbed to be local embeddings.  
Then each time we apply Akopyan's Theorem we can ensure that the total map is still an embedding.  
But for Corollary \ref{cor:isometry Minkowski space} we are not concerned with the map $h$ being an embedding, and so the map $f^*$ can be removed.  
This yields the appropriate amount of coordinates for Corollary \ref{cor:isometry Minkowski space}.

\end{proof}

\vskip 20pt

\begin{proof}[Proof of Theorem \ref{theorem:Nash-Kuiper into Minkowski space}]
Let $(M, G)$ denote an $n$-manifold with a metric tensor $G$ of any signature, 
and let $f: M \rightarrow \Rpq$ be any continuous map with $p, q \geq 2n$.  
Since there are at least $4n$ ambient dimensions, by Whitney we may assume that $f$ is a smooth embedding.
Note that we are using a capital $G$ instead of a lowercase $g$ as is used in the statement of Theorem \ref{theorem:Nash-Kuiper into Minkowski space} in order to be consistent with the notation in the proof of Theorem \ref{thm:isometric embedding Minkowski space}

Just as above, we decompose $f = f^+ \oplus f^-$ 
where $f^+: M \rightarrow \R^{p,0}$ and $f^- : M \rightarrow \R^{0,q}$.  
To remain consistent with notation, let $G_f$, $G_f^+$, and $G_f^-$ denote the pullback metrics induced by $f$, $f^+$, and $f^-$, respectively.   
It is well known (for example, see \cite{Nash2} or \cite{Greene}) that $G_f = G_f^+ + G_f^-$.  
Also, since the codomains of both $f^+$ and $f^-$ contain at least $2n$ dimensions, by Whitney we may assume that both maps are immersions.

Just as before, we construct a quadratic form $H$ on $M$ such that both $H < G$ and $H < G_f$.  
If $M$ is compact then we can simply obtain $H$ by scaling $Q$, the Euclidean quadratic form on $\R^{p+q}$, by a suitably large negative number.  
For $M$ non-compact essentially the same construction works.
Let $\{ C_i \}_{i=1}^{\infty}$ be a compact exhaustion of $M$, i.e. $\cup_{i=1}^\infty C_i = M$ and $C_i \subseteq C_{i+1}$ for all $i$.
Let $\a_i$ be a negative constant such that $\a_i < \a_{i-1}$, $\a_i Q < G$, and $\a_i Q < G_f$ all within $C_{i+1}$.
Then we require that $H \leq \a_i Q$ when restricted to the boundary $C_i$, and we use a smooth partition of unity to vary the quadratic form within $C_{i+1} \setminus C_i$.

Now that we have this form $H$, we proceed in exactly the same way as in the proof of Theorem \ref{thm:isometric embedding Minkowski space}.  
We again have that 
	\begin{equation*}
	G_f^+ + G_f^- = G_f > H	\qquad	\Longrightarrow		\qquad	G_f^- > H - G_f^+.
	\end{equation*}
and we can apply the Nash-Kuiper Theorem (in the negative-definite setting) to obtain a $C^1$-map $h^- :M \to \R^{0,q}$ such that $G_h^- = H - G_f^+$.
Two remarks:
	
\n (1)  In the construction of the Nash-Kuiper Theorem, the map $h^-$ is obtained as the limit of smooth maps whose induced metric converges to that of $h^-$.  
	So we may really assume that $h^-$ is a smooth map whose induced metric $G_h^-$ satisfies
		\begin{equation*}
		G_h^- \approx H - G_f^+		\qquad	\Longrightarrow	\qquad	G_f^+ + G_h^- \approx H
		\end{equation*}
	and where this approximation is as close as we like.
	
	\vskip 10pt
		
\n (2)  In order to apply the Nash-Kuiper $C^1$-isometric embedding Theorem to $f^-$, we need a unit normal vector field $\eta: f^-(M) \rightarrow \R^q$ 
	(see pg. 551 of \cite{Kuiper}).  
	If $f^-$ happened to be an embedding (which it may not be), then choosing fine enough iterations of this process would ensure that $h^-$ were also an embedding.
	But, clearly, the map $\eta \oplus \vec{0} : f(M) \rightarrow \Rpq$ is also a unit normal vector field to the image of $f$.  
	Then since $f = f^+ \oplus f^-$ is an embedding, applying small enough iterations of the Nash-Kuiper process (with respect to $\eta$) preserves the fact that $f^+ \oplus h^-$ is an embedding. 
	
\vskip 10pt
	
Now, just as above we have that
	\begin{equation*}
	G > H \approx G_f^+ + G_h^-	\qquad	\Longrightarrow		\qquad	G_f^+ < G - G_h^-.
	\end{equation*}
So we again apply the Nash-Kuiper $C^1$-isometric embedding Theorem to obtain a $C^1$ map $h^+ : M \rightarrow \R^{p,0}$ with associated quadratic form $G_h^+$ so that
	\begin{equation*}
	G_h^+ = G - G_h^-	\qquad	\Longrightarrow		\qquad	G = G_h^+ + G_h^- = G_h
	\end{equation*}
where $h = h^+ \oplus h^-$.  
By the same considerations as above we have that $h$ is an embedding, and is thus our desired $C^1$-isometric embedding.

\end{proof}

\begin{remark}
We needed both $p,q \geq 2n$ in Theorem \ref{theorem:Nash-Kuiper into Minkowski space} to ensure that both $f^+$ and $f^-$ could be perturbed to be immersions.  
But if either map is already an immersion to begin with, then we do not need such high codimension.  
In particular, the dimension requirements could be as low as $p,q \geq n+1$.  
Note that this guarantees at least $2n+2$ ambient dimensions, so there is still no issue with perturbing the total map $f$ to be an embedding.
\end{remark}

\vskip 10pt

\end{document}